\documentclass[11pt]{amsart}
\usepackage{amsmath,amssymb,amsthm,mathtools}
\usepackage[margin=1.1in]{geometry}
\usepackage{enumitem}
\usepackage{booktabs}
\usepackage[colorlinks=true,linkcolor=blue,citecolor=blue,urlcolor=blue]{hyperref}

\newtheorem{theorem}{Theorem}[section]
\newtheorem{lemma}[theorem]{Lemma}
\newtheorem{proposition}[theorem]{Proposition}
\newtheorem{corollary}[theorem]{Corollary}
\newtheorem{conjecture}[theorem]{Conjecture}
\newtheorem*{thmA}{Theorem A}
\newtheorem*{thmB}{Theorem B}
\newtheorem*{thmC}{Theorem C}
\newtheorem*{thmD}{Theorem D}
\newtheorem*{dsc}{Double Sequencing Conjecture}
\theoremstyle{definition}
\newtheorem{definition}[theorem]{Definition}
\newtheorem{example}[theorem]{Example}
\newtheorem{remark}[theorem]{Remark}
\newtheorem{question}[theorem]{Question}

\newcommand{\Z}{\mathbb{Z}}
\newcommand{\bg}{\mathbf{g}}
\newcommand{\bh}{\mathbf{h}}
\newcommand{\bp}{\mathbf{p}}
\newcommand{\bq}{\mathbf{q}}
\newcommand{\bc}{\mathbf{c}}
\newcommand{\bd}{\mathbf{d}}
\newcommand{\bz}{\mathbf{z}}
\newcommand{\bt}{\mathbf{t}}
\newcommand{\bx}{\mathbf{x}}

\newcommand{\Dseq}{D-sequence}

\makeatletter
\@namedef{subjclassname@2020}{\textup{2020} Mathematics Subject Classification}
\makeatother

\begin{document}

\title{Doubly sequenceable groups}

\author{Mohammad Javaheri}
\address{Department of Mathematics and Statistics, Siena University\\
515 Loudon Road, Loudonville, NY 12211, USA}
\email{mjavaheri@siena.edu}


\begin{abstract}
Given a sequence $\bg: g_0,\dots,g_m$ in a finite group $G$ with $g_0=1_G$, let
$\bar{\bg}:\bar g_0,\dots,\bar g_m$ be the sequence of consecutive quotients of $\bg$, defined by
$\bar g_0=1_G$ and $\bar g_i=g_{i-1}^{-1}g_i$ for $1\le i\le m$. We say that $G$ is
\emph{doubly sequenceable} if there exists a sequence $\bg$ in $G$ such that every element of $G$
appears exactly twice in each of $\bg$ and $\bar{\bg}$; such a sequence is \emph{cyclic} if its
last term is $1_G$. We show that every group that is abelian, of odd order, sequenceable,
R-sequenceable, or terraced has a cyclic double sequencing, and that a group is terraced if and
only if it has a palindromic cyclic double sequencing. We show that if $H$ is a group of odd order
or a sequenceable group and $K$ is an abelian group, then $H\times K$ has a cyclic double
sequencing. We also show that if $H$ has a cyclic double sequencing and $K$ is an abelian group
with no direct summand of order $2$, then $H\times K$ has a cyclic double sequencing, and we
extend this to terraced groups $H$ and abelian groups $K$ with at most six direct summands of order
$2$. Consequently, complete double Latin squares exist for every order, and the double sequencing
conjecture holds for all groups of order at most $87$ except possibly $64$, and for all
sufficiently large groups.
\end{abstract}

\maketitle

\section{Introduction}

Throughout, all groups are finite. Gordon \cite{Gor61} defined a group $G$ of order $n$ to be
\emph{sequenceable} if there exists a permutation $h_0,h_1,\dots,h_{n-1}$ of the elements of $G$
such that $h_0=1_G$ and the partial products $h_0, h_0h_1,\dots,h_0h_1\cdots h_{n-1}$ are pairwise
distinct. Equivalently, $G$ is sequenceable if there exists a permutation $g_0,\dots,g_{n-1}$ of the
elements of $G$ with $g_0=1_G$ such that the consecutive quotients
$1_G, g_0^{-1}g_1, g_1^{-1}g_2,\dots, g_{n-2}^{-1}g_{n-1}$ also form a permutation of the elements of
$G$; such a sequence is called a \emph{sequencing} of $G$. Gordon proved that an abelian group is
sequenceable if and only if it is \emph{binary}, i.e., it has a unique element of order $2$
\cite{Gor61,JK21}.

Gordon also observed that the dihedral groups $D_6$ and $D_8$ are not sequenceable. Isbell
\cite{Isb90} and Li \cite{Li97} showed that the dihedral groups $D_{2n}$ are sequenceable for all
$n\ge 5$, and Anderson and Ihrig \cite{AI93} proved that every solvable binary group except the
quaternion group $Q_8$ is sequenceable. Keedwell conjectured \cite{Kee83} that every non-abelian
group other than $D_6$, $D_8$ and $Q_8$ is sequenceable; Müyesser and Pokrovskiy \cite{MP25}
recently proved this for all sufficiently large groups. See \cite{Oll-survey} for a survey.

\subsection*{Double sequencings}
Given a sequence $\bg:g_0,\dots,g_m$ in a group $G$ with $g_0=1_G$, the sequence of
\emph{consecutive quotients} of $\bg$ is $\bar\bg:\bar g_0,\dots,\bar g_m$, where
\[
\bar g_0=1_G,\qquad \bar g_i=g_{i-1}^{-1}g_i\quad(1\le i\le m).
\]

\begin{definition}\label{def:D}
Let $G$ be a group of order $n$ and let $\bg:g_0,\dots,g_{2n-1}$ be a sequence in $G$ with
$g_0=1_G$. We say that $\bg$ is a \emph{double sequencing} of $G$ (a \emph{\Dseq} for short) if
every element of $G$ appears exactly twice in each of $\bg$ and $\bar\bg$. If a \Dseq{} of $G$
exists, $G$ is \emph{doubly sequenceable} (\emph{D-sequenceable} for short). A \Dseq{} is
\emph{cyclic} if $g_{2n-1}=1_G$, and \emph{palindromic} if $g_{2n-1-i}=g_i$ for $0\le i\le 2n-1$.
More generally, for $k\ge 1$, a \emph{$k$-fold sequencing} of $G$ is a sequence
$g_0=1_G,g_1,\dots,g_{kn-1}$ such that every element of $G$ appears exactly $k$ times in each of
$\bg$ and $\bar\bg$.
\end{definition}

The correspondence $\bg\leftrightarrow\bar\bg$ is one-to-one, since $g_0=1_G$ and
$g_i=\bar g_1\bar g_2\cdots\bar g_i$ for $i\ge1$. Hence $G$ is D-sequenceable if and only if there
is a sequence $\bh:h_0=1_G,h_1,\dots,h_{2n-1}$ such that every element of $G$ appears twice as a
term of $\bh$ and twice as a partial product of $\bh$.

\subsection*{Double Latin squares}
A classical application of sequencings is the construction of complete Latin squares
\cite{Gor61}. A \emph{double Latin square} of order $n$ is a $2n\times 2n$ array on $n$ symbols in
which each symbol appears exactly twice in each row and exactly twice in each column
\cite{AH80a,AH80b,HMRN03}. We call it \emph{row-complete} if for every ordered pair $(\alpha,\beta)$
of \emph{distinct} symbols the equations
\begin{equation}\label{eq:rowcomplete}
L_{i,j}=\alpha,\qquad L_{i,j+1}=\beta
\end{equation}
have exactly four solutions $(i,j)$ with $0\le i\le 2n-1$ and $0\le j\le 2n-2$;
\emph{column-complete} is defined similarly with the roles of rows and columns exchanged, and
\emph{complete} means row- and column-complete. (Four is the largest number compatible with
uniformity: there are $2n(2n-1)$ horizontally adjacent pairs of cells, and each symbol $\alpha$
occupies $4n$ cells, at most $2\cdot 2n$ of which are followed in their row by another cell.)

\begin{proposition}\label{prop:latin}
Let $\bg$ be a \Dseq{} of a group $G$ of order $n$, and let $L$ be the $2n\times2n$ array
$L_{i,j}=g_i^{-1}g_j$ $(0\le i,j\le 2n-1)$. Then $L$ is a complete double Latin square. Moreover,
for every symbol $\alpha$ the equations \eqref{eq:rowcomplete} with $\beta=\alpha$ have exactly
two solutions. If $\bg$ is cyclic, then in the array $M_{i,j}=g_i^{-1}g_{j+1}$ (indices modulo
$2n$) every element of $G$ appears exactly twice in each row, in each column, and on the main
diagonal.
\end{proposition}

\begin{proof}
Row $i$ of $L$ is the sequence $g_i^{-1}\bg$, a left translate of $\bg$, and column $j$ consists
of the elements $g_i^{-1}g_j$, which run through $\{x^{-1}g_j: x\in G\}$ twice. Hence $L$ is a
double Latin square.

Let $\alpha,\beta\in G$. A pair $(i,j)$ solves \eqref{eq:rowcomplete} if and only if
$g_i^{-1}g_j=\alpha$ and $\bar g_{j+1}=g_j^{-1}g_{j+1}=\alpha^{-1}\beta$. Since $\bar g_0=1_G$, the
number of indices $0\le j\le 2n-2$ with $\bar g_{j+1}=\alpha^{-1}\beta$ is $2$ if $\alpha\ne\beta$
and $1$ if $\alpha=\beta$; for each such $j$ there are exactly two indices $i$ with
$g_i=g_j\alpha^{-1}$. This gives four, respectively two, solutions. Similarly,
$L_{i,j}=\alpha$, $L_{i+1,j}=\beta$ holds if and only if $g_j=g_i\alpha$ and
$\bar g_{i+1}=g_i^{-1}g_{i+1}=\alpha\beta^{-1}$, and the same count applies; so $L$ is complete.

If $\bg$ is cyclic, row $i$ of $M$ is $g_i^{-1}(g_1,\dots,g_{2n-1},g_0)$, column $j$ is
$\{g_i^{-1}g_{j+1}\}_i$, and the diagonal is $\bar g_1,\dots,\bar g_{2n-1},g_{2n-1}^{-1}g_0$, where
$g_{2n-1}^{-1}g_0=1_G=\bar g_0$. Each of these lists contains every element twice.
\end{proof}

\begin{dsc}
Every finite group is doubly sequenceable.
\end{dsc}

\subsection*{Main results}
A group is \emph{R-sequenceable} if there is a permutation $g_1,\dots,g_{n-1}$ of the
non-identity elements of $G$ such that $g_1^{-1}g_2,\dots,g_{n-2}^{-1}g_{n-1},g_{n-1}^{-1}g_1$ is also
a permutation of the non-identity elements \cite{FGM78,AKP17}. A \emph{terrace} \cite{Bai84} of $G$
is an ordering $g_0,\dots,g_{n-1}$ of the elements of $G$ such that, with
$\bar g_i=g_{i-1}^{-1}g_i$ for $1\le i\le n-1$,
\begin{align*}
&\#\{1\le i\le n-1:\ \bar g_i=\alpha\}=1 &&\text{if } \alpha^2=1_G\neq\alpha,\\
&\#\{1\le i\le n-1:\ \bar g_i\in\{\alpha,\alpha^{-1}\}\}=2 &&\text{if } \alpha^2\neq 1_G.
\end{align*}
Terraced groups are also called \emph{2-sequenceable}. Every sequencing is a terrace. Bailey
conjectured that every group except the elementary abelian $2$-groups of order at least $4$ is
terraced \cite{Bai84}; this is known for abelian groups (see \cite{Oll12,Oll-survey}), for groups of odd order
\cite{AI92}, and for all groups of order at most $87$ except possibly $64$ \cite{And92} (see also
\cite{Oll19}).

\begin{thmA}
Each of the following groups has a cyclic \Dseq:
\begin{enumerate}[label=\textup{(\alph*)}]
\item sequenceable groups and R-sequenceable groups;
\item terraced groups; in fact, a group is terraced if and only if it has a palindromic cyclic
\Dseq;
\item abelian groups;
\item groups of odd order;
\item $D_6$, $D_8$ and $Q_8$.
\end{enumerate}
Moreover, if a group $G$ has a $k$-fold sequencing and $N$ is a normal subgroup of order $r$, then
$G/N$ has a $kr$-fold sequencing.
\end{thmA}

The cyclic versions are needed for the product constructions below. In view of
\cite{MP25,And92}, Theorem~A has the following consequence.

\begin{corollary}\label{cor:status}
The double sequencing conjecture holds for all groups of order at most $87$ except possibly
$64$, and for all groups of sufficiently large order. For every $n\ge1$ there exists a complete
double Latin square of order $n$.
\end{corollary}

Our next results concern direct products. Every abelian group $K$ can be written as
$K\cong(\Z_2)^k\times K'$, where $K'$ has no direct summand of order $2$; here $k$ is the number of
cyclic factors of order exactly $2$ in the primary decomposition of $K$.

\begin{thmB}
Let $H$ be a group of odd order or a sequenceable group, and let $K$ be an abelian group. Then
$H\times K$ has a cyclic \Dseq.
\end{thmB}

\begin{thmC}
Let $H$ be a group with a cyclic \Dseq{} \textup{(}for instance, a terraced or an R-sequenceable
group\textup{)}, and let $K$ be an abelian group with no direct summand of order $2$. Then
$H\times K$ has a cyclic \Dseq.
\end{thmC}

\begin{thmD}
Let $H$ be a terraced group and let $K\cong(\Z_2)^k\times K'$ be an abelian group, where $K'$ has
no direct summand of order $2$ and $k\le 6$. Then $H\times K$ has a cyclic \Dseq.
\end{thmD}

The restriction $k\le 6$ in Theorem~D comes from the existence of certain combinatorial objects
in $(\Z_2)^k$, which we call \emph{frames} (Definition~\ref{def:frame}); we found frames for
$k\le 6$ by computer and conjecture that they exist for all $k$ (Conjecture~\ref{conj:frames}), in
which case Theorem~D holds for every abelian group $K$.

\subsection*{Organization}
In Section~\ref{sec:basic} we collect basic constructions and prove Theorem~A and
Corollary~\ref{cor:status}. In Section~\ref{sec:lift} we show that a cyclic \Dseq{} of $H$ lifts to
a cyclic \Dseq{} of $H\times\Z_m$ whenever $m\not\equiv 2\pmod 4$, which proves Theorem~C. In
Section~\ref{sec:Z2} we treat direct factors $(\Z_2)^k$, and in Section~\ref{sec:proofs} we prove
Theorems~B and~D. Section~\ref{sec:open} contains open problems.

\section{Basic constructions}\label{sec:basic}

\subsection{Circular sequences}
It is often convenient to regard a sequence $\bx=(x_0,\dots,x_{L-1})$ as \emph{circular}, with
indices read modulo $L$; its \emph{circular quotients} are $x_{i-1}^{-1}x_i$ for $i\in\Z_L$
(including $x_{L-1}^{-1}x_0$).

\begin{lemma}\label{lem:circular}
Let $G$ be a group of order $n$ and let $\bx=(x_0,\dots,x_{2n-1})$ be a circular sequence in $G$ in
which every element of $G$ appears exactly twice, and in which every element of $G$ appears exactly
twice among the $2n$ circular quotients $x_{i-1}^{-1}x_i$, $i\in\Z_{2n}$. Then there exists
$j\in\Z_{2n}$ with $x_{j-1}=x_j$, and $g_i=x_j^{-1}x_{j+i}$ $(0\le i\le 2n-1)$ defines a cyclic
\Dseq{} of $G$. Conversely, every cyclic \Dseq{} is such a circular sequence.
\end{lemma}

\begin{proof}
Since $1_G$ is a circular quotient (twice), there is $j$ with $x_{j-1}^{-1}x_j=1_G$. Left
multiplication by $x_j^{-1}$ preserves both the multiplicities of the terms and all circular
quotients. Hence $g_0=1_G$, $g_{2n-1}=x_j^{-1}x_{j-1}=1_G$, every element appears twice in $\bg$,
and $\bar g_1,\dots,\bar g_{2n-1}$ are the circular quotients of $\bx$ other than
$x_{j-1}^{-1}x_j=1_G$, which is replaced by $\bar g_0=1_G$. The converse holds because for a cyclic
\Dseq{} the extra circular quotient $g_{2n-1}^{-1}g_0=1_G$ plays the role of $\bar g_0$.
\end{proof}

\subsection{Sequenceable groups}

\begin{lemma}\label{lem:kfold}
Let $G$ be a sequenceable group of order $n$ and $k\ge1$. Then $G$ has a $k$-fold sequencing. For
$k=2$ it can be chosen to be a cyclic, palindromic \Dseq.
\end{lemma}

\begin{proof}
Let $\bg:g_0=1_G,\dots,g_{n-1}$ be a sequencing of $G$, and write $\bg^{\mathrm{rev}}$ for
$g_{n-1},\dots,g_1,g_0$. Let $\bp$ be the concatenation of $k$ blocks
$\bg,\bg^{\mathrm{rev}},\bg,\bg^{\mathrm{rev}},\dots$; explicitly $p_{tn+j}=g_j$ if $t$ is even
and $p_{tn+j}=g_{n-1-j}$ if $t$ is odd ($0\le j\le n-1$, $0\le t\le k-1$). Each element appears once
in each block, hence $k$ times in $\bp$. The quotient sequence is
\[
\bar\bp:\ \underbrace{1_G,\bar g_1,\dots,\bar g_{n-1}}_{t=0},\ \
\underbrace{1_G,\bar g_{n-1}^{-1},\dots,\bar g_1^{-1}}_{t=1},\ \
\underbrace{1_G,\bar g_1,\dots,\bar g_{n-1}}_{t=2},\ \dots,
\]
because at each junction a block ends and the next begins with the same element, and inside a
reversed block $g_{j}^{-1}g_{j-1}=\bar g_j^{-1}$. Since $\{\bar g_1,\dots,\bar g_{n-1}\}$ and
$\{\bar g_1^{-1},\dots,\bar g_{n-1}^{-1}\}$ both equal $G\setminus\{1_G\}$, every element appears
exactly $k$ times in $\bar\bp$. For $k=2$, $\bp=(\bg,\bg^{\mathrm{rev}})$ ends with $g_0=1_G$ and is
palindromic.
\end{proof}

\subsection{Terraces}
If $\bg$ is a terrace and $x\in G$, then $x\bg$ is a terrace with the same quotients; hence we may
always assume that a terrace begins with $1_G$.

\begin{proposition}\label{prop:terrace}
A group $G$ of order $n$ is terraced if and only if it has a palindromic cyclic \Dseq. More
precisely, if $\bg:g_0=1_G,\dots,g_{n-1}$ is a terrace, then
$\bg\,\bg^{\mathrm{rev}}: g_0,\dots,g_{n-1},g_{n-1},\dots,g_0$ is a palindromic cyclic \Dseq,
and every palindromic \Dseq{} arises in this way.
\end{proposition}

\begin{proof}
Let $\bp=\bg\,\bg^{\mathrm{rev}}$. Each element appears twice in $\bp$, and as in the proof of
Lemma~\ref{lem:kfold},
$\bar\bp: 1_G,\bar g_1,\dots,\bar g_{n-1},1_G,\bar g_{n-1}^{-1},\dots,\bar g_1^{-1}$. The identity
appears exactly twice (since $\bar g_i\ne 1_G$ for $i\ge1$). For $\alpha\ne1_G$, the number of
occurrences of $\alpha$ in $\bar\bp$ equals
$\#\{i\ge1:\bar g_i=\alpha\}+\#\{i\ge1:\bar g_i=\alpha^{-1}\}$, which is $2$ by the terrace
condition (if $\alpha^2=1_G$ both summands equal $1$). Hence $\bp$ is a palindromic cyclic
\Dseq.

Conversely, let $\bp$ be a palindromic \Dseq{} of length $2n$; then $p_{2n-1}=p_0=1_G$. If
$p_i=x$ then also $p_{2n-1-i}=x$, and $i\ne 2n-1-i$; since $x$ appears exactly twice, it appears
exactly once among $p_0,\dots,p_{n-1}$. So $\bg=(p_0,\dots,p_{n-1})$ is an ordering of $G$ and
$\bp=\bg\,\bg^{\mathrm{rev}}$, and the computation above shows that $\bg$ satisfies the terrace
conditions.
\end{proof}

\subsection{R-sequenceable groups}

\begin{lemma}\label{lem:R}
Let $G$ be an R-sequenceable group of order $n\ge3$ and let $k\ge2$. Then $G$ has a $k$-fold
sequencing whose last term is $1_G$; in particular $G$ has a cyclic \Dseq.
\end{lemma}

\begin{proof}
Let $g_1,\dots,g_{n-1}$ be an R-sequencing, with indices read modulo $n-1$ in $\{1,\dots,n-1\}$,
and let $Q=\{g_j^{-1}g_{j+1}: j\in\Z_{n-1}\}=G\setminus\{1_G\}$. Choose $r,s$ with
\[
g_r=g_{n-1}^{-1},\qquad g_s=g_{n-1}^{-1}g_{n-2};
\]
both exist because the right-hand sides are not the identity, and $r\ne s$ because $g_{n-2}\ne1_G$.
Let $\ell_0=(g_{s+1},g_{s+2},\dots,g_{s+n-1})$, the R-sequencing read cyclically starting after
$g_s$ (so $\ell_0$ ends with $g_s$), and let $\ell_1$ be obtained from $\ell_0$ by writing the term
$g_r$ twice in succession. Define
\[
\bq:\ 1_G,\ g_{n-1}^{-1}g_1,\ g_{n-1}^{-1}g_2,\dots,\ g_{n-1}^{-1}g_{n-2},\ \ell_1,\ 1_G,
\]
a sequence of length $1+(n-2)+n+1=2n$.

\emph{Terms.} The map $x\mapsto g_{n-1}^{-1}x$ is a bijection from $G\setminus\{1_G\}$ onto
$G\setminus\{g_{n-1}^{-1}\}$ sending $g_{n-1}$ to $1_G$; hence the first $n-1$ terms of $\bq$ contain
every element of $G$ other than $g_r$ exactly once. The list $\ell_1$ contains every non-identity
element once and $g_r$ once more, and the last term is $1_G$. So every element appears twice in
$\bq$.

\emph{Quotients.} Besides $\bar q_0=1_G$, the quotients inside the first $n-1$ terms are
$g_{n-1}^{-1}g_1$ and $g_{j}^{-1}g_{j+1}$ for $1\le j\le n-3$, i.e., all elements of $Q$ except
$g_{n-2}^{-1}g_{n-1}=g_s^{-1}$. The quotient at the junction with $\ell_1$ is
$(g_{n-1}^{-1}g_{n-2})^{-1}g_{s+1}=g_s^{-1}g_{s+1}$. The quotients inside $\ell_1$ are the
elements $g_j^{-1}g_{j+1}$ of $Q$ other than $g_s^{-1}g_{s+1}$, together with one $1_G$ coming from
the repeated $g_r$. The last quotient is $g_s^{-1}1_G=g_s^{-1}$. Altogether $\bar\bq$ consists of
$Q$ twice and $1_G$ twice, so $\bq$ is a cyclic \Dseq.

For $k\ge3$, let $\bp$ be obtained from $\bq$ by inserting $k-2$ copies of $\ell_0$ between the
last two terms of $\bq$ and then appending $k-2$ further copies of $1_G$. Each copy of $\ell_0$
contributes every non-identity element once, the appended identities contribute $1_G$ $k-2$
times, and the length is $kn$. The quotient entering each copy of $\ell_0$ (from a term $g_s$) is
$g_s^{-1}g_{s+1}$, so each copy contributes the quotients $Q$ once; the quotient $g_s^{-1}$ now
occurs where the last copy of $\ell_0$ meets $1_G$, and the appended identities contribute $k-2$
quotients $1_G$. Hence $\bp$ is a $k$-fold sequencing ending in $1_G$.
\end{proof}

\begin{example}\label{ex:Z5}
In $\Z_5$, $1,3,4,2$ is an R-sequencing (more generally, if $a$ is a primitive root modulo an odd
prime $p$, then $1,a,a^2,\dots,a^{p-2}$ is an R-sequencing of $\Z_p$). Here $r=2$, $s=4$,
$\ell_0:1,3,4,2$, $\ell_1:1,3,3,4,2$, and Lemma~\ref{lem:R} gives the cyclic \Dseq{}
$0,4,1,2,1,3,3,4,2,0$ of $\Z_5$.
\end{example}

\subsection{Quotient groups and symmetric sequencings}

\begin{lemma}\label{lem:quotient}
Let $\bg$ be a $k$-fold sequencing of $G$, let $N\trianglelefteq G$ with $|N|=r$, and let
$\pi:G\to G/N$ be the natural map. Then $\pi(\bg)$ is a $kr$-fold sequencing of $G/N$; it ends with
the identity if and only if the last term of $\bg$ lies in $N$. In particular, if $G$ is
sequenceable and $|N|=2$, then $G/N$ is D-sequenceable.
\end{lemma}

\begin{proof}
Each coset $xN$ contains $r$ elements, each appearing $k$ times in $\bg$, and
$\overline{\pi(\bg)}=\pi(\bar\bg)$ because $\pi$ is a homomorphism.
\end{proof}

Let $G$ be a binary group of order $n=2m$ with involution $z$; then $z$ is central. A sequencing
$\bg$ of $G$ is \emph{symmetric} if $\bar g_m=z$ and $\bar g_{m+i}=\bar g_{m-i}^{-1}$ for
$1\le i\le m-1$ \cite{AI93,Oll-survey}. Anderson and Ihrig \cite{AI93} proved that every solvable
binary group other than $Q_8$ has a symmetric sequencing.

\begin{proposition}\label{prop:symmetric}
Let $G$ be a binary group of order $2m$ with involution $z$, and let $\bg$ be a symmetric
sequencing of $G$. Then $g_{m+i}=z\,g_{m-1-i}$ for $0\le i\le m-1$. Consequently, the projection of
$(g_0,\dots,g_{m-1})$ to $G/\langle z\rangle$ is a terrace, and $G/\langle z\rangle$ has a
palindromic cyclic \Dseq.
\end{proposition}

\begin{proof}
We have $g_m=g_{m-1}\bar g_m=g_{m-1}z=zg_{m-1}$, and if $g_{m+i-1}=zg_{m-i}$ for some
$1\le i\le m-1$, then
$g_{m+i}=g_{m+i-1}\bar g_{m+i}=zg_{m-i}\bar g_{m-i}^{-1}=zg_{m-i-1}$. Thus $\pi(\bg)$ is
palindromic, and by Lemma~\ref{lem:quotient} it is a \Dseq{} of $G/\langle z\rangle$; it is cyclic
since $\pi(g_{2m-1})=\pi(zg_0)$ is the identity. Proposition~\ref{prop:terrace} completes the
proof.
\end{proof}

\begin{corollary}\label{cor:odd}
Every group of odd order is terraced and has a cyclic \Dseq.
\end{corollary}

\begin{proof}
Let $H$ have odd order. By the Feit--Thompson theorem $H$ is solvable, so $H\times\Z_2$ is a
solvable binary group different from $Q_8$, and it has a symmetric sequencing by \cite{AI93}.
Apply Proposition~\ref{prop:symmetric} with $z=(1_H,1)$. (That odd groups are terraced is also
proved directly in \cite{AI92}.)
\end{proof}

\subsection{Abelian groups and small groups}

\begin{theorem}\label{thm:abelian}
Every abelian group has a cyclic \Dseq.
\end{theorem}

\begin{proof}
For $|G|\le 2$ take $0,0$ and $0,1,1,0$. Let $|G|\ge3$. If $G$ is binary, it is sequenceable
\cite{Gor61}, and Lemma~\ref{lem:kfold} applies. Otherwise the Sylow $2$-subgroup of $G$ is trivial
or non-cyclic, and $G$ is R-sequenceable by the theorem of Alspach, Kreher and Pastine
\cite{AKP17} (the Friedlander--Gordon--Miller conjecture \cite{FGM78}); Lemma~\ref{lem:R}
applies.
\end{proof}

\begin{example}\label{ex:small}
Write $D_{2m}=\langle\rho,\sigma\mid \rho^m=\sigma^2=1,\ \sigma\rho\sigma=\rho^{-1}\rangle$, so that
$\rho^r\sigma^a\cdot\rho^s\sigma^b=\rho^{r+(-1)^a s}\sigma^{a+b}$, and $Q_8=\{\pm1,\pm i,\pm j,\pm k\}$.
The groups $D_6$, $D_8$, $Q_8$ are not sequenceable, but they are terraced \cite{Bai84,And92}:
\begin{center}
\begin{tabular}{@{}lll@{}}
\toprule
group & terrace $\bg$ & quotients $\bar g_1,\dots,\bar g_{n-1}$\\\midrule
$D_6$ & $1,\ \rho,\ \sigma,\ \rho^2,\ \rho^2\sigma,\ \rho\sigma$ &
$\rho,\ \rho^2\sigma,\ \rho\sigma,\ \sigma,\ \rho$\\
$D_8$ & $1,\ \rho,\ \sigma,\ \rho^3,\ \rho\sigma,\ \rho^3\sigma,\ \rho^2\sigma,\ \rho^2$ &
$\rho,\ \rho^3\sigma,\ \rho\sigma,\ \rho^2\sigma,\ \rho^2,\ \rho,\ \sigma$\\
$Q_8$ & $1,\ i,\ j,\ k,\ -k,\ -i,\ -j,\ -1$ & $i,\ -k,\ -i,\ -1,\ -j,\ -k,\ -j$\\
\bottomrule
\end{tabular}
\end{center}
In each case every involution occurs once among the quotients and each pair
$\{\alpha,\alpha^{-1}\}$ with $\alpha^2\ne1$ is hit twice. By Proposition~\ref{prop:terrace}, each
terrace $\bg$ yields the cyclic \Dseq{} $\bg\,\bg^{\mathrm{rev}}$. Non-palindromic cyclic
\Dseq{}s also exist; for instance, for $D_6$,
\[
1,\ \rho,\ \rho,\ \sigma,\ \rho^2,\ \rho\sigma,\ \rho^2\sigma,\ \rho\sigma,\ \rho^2\sigma,\ \rho^2,\
\sigma,\ 1 .
\]
\end{example}

\begin{proof}[Proof of Theorem~A and Corollary~\ref{cor:status}]
Part (a) is Lemmas~\ref{lem:kfold} and~\ref{lem:R} (groups of order $\le2$ are abelian), part (b)
is Proposition~\ref{prop:terrace}, part (c) is Theorem~\ref{thm:abelian}, part (d) is
Corollary~\ref{cor:odd}, and part (e) is Example~\ref{ex:small}; the last statement is
Lemma~\ref{lem:quotient}.

For Corollary~\ref{cor:status}: every group of order at most $87$, other than those of order $64$,
is either an elementary abelian $2$-group (hence abelian) or terraced \cite{And92}. By
\cite{MP25} there is $n_0$ such that every non-abelian group of order at least $n_0$ is
sequenceable; together with Theorem~\ref{thm:abelian} this settles all groups of order at least
$n_0$. Finally, $\Z_n$ has a \Dseq{} for every $n$, and Proposition~\ref{prop:latin} produces a
complete double Latin square of order $n$.
\end{proof}

\section{\texorpdfstring{Lifting by $\Z_m$, $m\not\equiv2\pmod 4$}{Lifting by Zm, m not 2 mod 4}}\label{sec:lift}

Let $H$ be a group of order $n$, $K$ a group of order $\kappa$, $G=H\times K$, and
$\pi:G\to H$ the projection. If $\bh:h_0,\dots,h_{2n-1}$ is a \Dseq{} of $H$ and
$\bg:g_0,\dots,g_{2\kappa n-1}$ is a \Dseq{} of $G$, we say that $\bg$ is a \emph{lift} of $\bh$
if $\pi(g_{i+2tn})=h_i$ for all $0\le i\le 2n-1$ and $0\le t\le\kappa-1$.

\begin{lemma}\label{lem:oddlift}
Let $m$ be odd. Every cyclic \Dseq{} of $H$ has a cyclic \Dseq{} lift in $H\times\Z_m$.
\end{lemma}

\begin{proof}
Let $\bh$ be a cyclic \Dseq{} of $H$ of length $2n$. For $0\le i\le 2mn-1$ write $i=2qn+r$ with
$0\le q\le m-1$ and $0\le r\le 2n-1$, and set
\[
g_i=\begin{cases}(h_r,-q)&r\text{ even},\\ (h_r,q+1)& r\text{ odd},\end{cases}
\qquad\text{so that}\qquad
\bar g_i=\begin{cases}(\bar h_r,-2q)&r=0,\\ (\bar h_r,2q+1)& r\text{ odd},\\
(\bar h_r,-2q-1)& r\neq0\text{ even}.\end{cases}
\]
(For $r=0$, $q\ge1$ we used $h_{2n-1}=1_H$, so that $g_{i-1}=(1_H,q)$; for $i=0$ the formula agrees
with $\bar g_0=(1_H,0)$.)

\emph{Terms.} Let $(h,t)\in H\times\Z_m$ and let $r\ne s$ be the two indices with $h_r=h_s=h$.
For $r$ even take $q\equiv -t$, for $r$ odd take $q\equiv t-1\pmod m$, and similarly for $s$.
This gives two distinct indices $i$ with $g_i=(h,t)$; by counting, every element appears exactly
twice.

\emph{Quotients.} Let $(u,v)\in H\times\Z_m$ and let $r\ne s$ be the two indices with
$\bar h_r=\bar h_s=u$. For each of $r,s$ choose $q\in\{0,\dots,m-1\}$ as follows: if the index is
$0$, solve $-2q\equiv v$; if it is odd, solve $2q+1\equiv v$; if it is even and nonzero, solve
$-2q-1\equiv v\pmod m$. All three congruences are solvable because $2$ is invertible modulo $m$.
This gives two distinct indices $i=2qn+r$ and $j=2q'n+s$ with $\bar g_i=\bar g_j=(u,v)$; by
counting, every element appears exactly twice in $\bar\bg$. (Note that for $u=1_H$ one of the two
indices is $0$, but the other one need not be $2n-1$.) Finally $g_{2mn-1}=(h_{2n-1},m)=(1_H,0)$,
so the lift is cyclic.
\end{proof}

\begin{example}\label{ex:Z15}
Applying Lemma~\ref{lem:oddlift} with $m=3$ to the \Dseq{} $0,4,1,2,1,3,3,4,2,0$ of $\Z_5$ gives
the following cyclic \Dseq{} of $\Z_5\times\Z_3$, where we write $\binom{b}{a}$ for $(a,b)$:
\[
\tbinom00\tbinom14\tbinom01\tbinom12\tbinom01\tbinom13\tbinom03\tbinom14\tbinom02\tbinom10\ \
\tbinom20\tbinom24\tbinom21\tbinom22\tbinom21\tbinom23\tbinom23\tbinom24\tbinom22\tbinom20\ \
\tbinom10\tbinom04\tbinom11\tbinom02\tbinom11\tbinom03\tbinom13\tbinom04\tbinom12\tbinom00 .
\]
\end{example}

We now turn to $\Z_{2^\nu}$ with $\nu\ge2$. We formulate the construction for an arbitrary abelian
group $E$ (written additively), since this isolates exactly what is needed.

\begin{definition}\label{def:perp}
Let $E$ be an abelian group. For permutations $\alpha,\beta$ of $E$ we write $\alpha\perp\beta$ if
the multiset $\{\alpha(x)-x: x\in E\}\cup\{\beta(x)-x:x\in E\}$ contains every element of $E$
exactly twice. A triple $(\alpha,\beta,\gamma)$ of permutations of $E$ is \emph{admissible} if
\begin{enumerate}[label=\textup{(\roman*)}]
\item $\alpha\perp\alpha^{-1}$;
\item $\beta\perp\gamma$ and $\gamma(0)=0$;
\item $\gamma\beta$ and $\gamma\alpha^{-1}\beta\alpha$ are $|E|$-cycles (cyclic permutations of
$E$ with no fixed points).
\end{enumerate}
\end{definition}

\begin{remark}\label{rem:noZ2}
$\Z_2$ has no admissible triple: if $\alpha$ is the identity of $\Z_2$, then $0$ appears four times
in condition (i), and if $\alpha$ is the transposition, then $1$ does. This is the reason why
direct factors $\Z_2$ require separate treatment (Section~\ref{sec:Z2}).
\end{remark}

\begin{lemma}\label{lem:triple}
For every $\nu\ge2$ the group $\Z_{2^\nu}$ has an admissible triple.
\end{lemma}

\begin{proof}
If $\nu=2$, let $\beta(x)=1-x$, $\gamma(x)=-x$ and let $\alpha$ be the permutation
$0\mapsto1\mapsto3\mapsto0$, $2\mapsto2$. Then $\{\alpha(x)-x\}=\{1,2,0,1\}$ and
$\{\alpha^{-1}(x)-x\}=\{3,0,3,2\}$, so $\alpha\perp\alpha^{-1}$; $\{\beta(x)-x\}=\{1,3,1,3\}$ and
$\{\gamma(x)-x\}=\{0,2,0,2\}$, so $\beta\perp\gamma$; and $\gamma\beta(x)=x-1$,
$\gamma\alpha^{-1}\beta\alpha(x)=x+1$ are $4$-cycles.

Let $\nu\ge3$, $M=2^\nu$, $\beta(x)=2^{\nu-2}-1-x$ and $\gamma(x)=-x$. Then $\gamma(0)=0$,
$\beta(x)-x=2^{\nu-2}-1-2x$ takes every odd value twice and $\gamma(x)-x=-2x$ every even value
twice, so $\beta\perp\gamma$. Define $\alpha$ on $0\le x\le 2^{\nu-1}-1$ by
\begin{equation}\label{eq:alpha}
\alpha(x)=\begin{cases}
2x+1 & 0\le x\le 2^{\nu-3}-1,\\
2x+2^{\nu-2}+1 & 2^{\nu-3}\le x\le 3\cdot2^{\nu-3}-1,\\
2x+2^{\nu-1}+1 & 3\cdot 2^{\nu-3}\le x\le 2^{\nu-1}-1,
\end{cases}
\end{equation}
and on $2^{\nu-1}\le x\le M-1$ by $\alpha(x)=2^{\nu-2}-1-\alpha(M-1-x)$ (all modulo $M$).

\emph{$\alpha$ is a permutation.} On the three ranges in \eqref{eq:alpha}, $\alpha$ maps
increasingly onto the odd residues in $[1,2^{\nu-2}-1]$, $[2^{\nu-1}+1,M-1]$ and
$[2^{\nu-2}+1,2^{\nu-1}-1]$, respectively; so the first half maps bijectively onto the odd
residues. Since $2^{\nu-2}-1$ is odd and $x\mapsto M-1-x$ swaps the two halves, the second half maps
bijectively onto the even residues.

\emph{Property (iii).} The defining relation gives $\alpha(-1-x)=2^{\nu-2}-1-\alpha(x)$ for all $x$
(for $x$ in the first half this is the definition, and for $x$ in the second half it is the
definition applied to $-1-x$). Hence $\beta\alpha(x)=\alpha(-1-x)$, so
$\gamma\alpha^{-1}\beta\alpha(x)=\gamma(-1-x)=x+1$, an $M$-cycle. Also
$\gamma\beta(x)=x+1-2^{\nu-2}$ is an $M$-cycle because $1-2^{\nu-2}$ is odd.

\emph{Property (i).} Let $A=\{\alpha(x)-x\}$ and note that $\{\alpha^{-1}(y)-y\}=\{x-\alpha(x)\}=-A$.
The multiset $A\cup(-A)$ is invariant under $v\mapsto -v$; hence, since it has $2M$ elements, it
suffices to show that each $v$ with $0\le v\le 2^{\nu-1}$ occurs at least twice. Put
$\tau=2^{\nu-3}$. The following solutions are checked directly from \eqref{eq:alpha}:
\begin{itemize}
\item $0\le v\le\tau-1$: $\alpha(x)-x=v$ for $x=v-1$ (for $v=0$, $x=2^{\nu-1}-1$), and $x-\alpha(x)=v$ for
$x=2^{\nu-1}-1-v$;
\item $v=\tau$: $\alpha(x)-x=v$ for $x=\tau-1$ and for $x=7\tau$;
\item $\tau+1\le v\le 2\tau-1$: $\alpha(x)-x=v$ for $x=6\tau+v$, and $x-\alpha(x)=v$ for $x=M-v$;
\item $2\tau\le v\le 3\tau-1$: $\alpha(x)-x=v$ for $x=v+2\tau$, and $x-\alpha(x)=v$ for $x=M-v$;
\item $v=3\tau$: $x-\alpha(x)=v$ for $x=3\tau-1$ and for $x=5\tau$;
\item $3\tau+1\le v\le4\tau$: $\alpha(x)-x=v$ for $x=v-2\tau-1$, and $x-\alpha(x)=v$ for
$x=6\tau-1-v$.
\end{itemize}
Thus $\alpha\perp\alpha^{-1}$.
\end{proof}

For a sequence of permutations $\mu_i$ we write $\prod_{j=i}^{k}\mu_j=\mu_k\mu_{k-1}\cdots\mu_i$
(composition, rightmost first).

\begin{lemma}\label{lem:mus}
Let $E$ be an abelian group with an admissible triple $(\alpha,\beta,\gamma)$, let $\ell\ge1$, and
let $\sigma$ be a fixed-point-free involution of $\{0,\dots,2\ell-1\}$. Then there are permutations
$\mu_0,\dots,\mu_{2\ell-1}$ of $E$ such that
\begin{enumerate}[label=\textup{(\roman*)}]
\item $\mu_i\perp\mu_{\sigma(i)}$ for all $i$, and $\mu_0=\gamma$;
\item $\prod_{j=1}^{2\ell}\mu_j=\mu_{2\ell}\mu_{2\ell-1}\cdots\mu_1$ is an $|E|$-cycle, where
$\mu_{2\ell}=\mu_0$.
\end{enumerate}
\end{lemma}

\begin{proof}
Put $\mu_0=\gamma$, $\mu_{\sigma(0)}=\beta$, and
$S=\{i: 0<i<\sigma(0)<\sigma(i)\}$. Choose $S'\subseteq S$ with $|S'|=\lceil|S|/2\rceil$. For
$i\in S$ put $\mu_i=\alpha$, $\mu_{\sigma(i)}=\alpha^{-1}$ if $i\in S'$, and $\mu_i=\alpha^{-1}$,
$\mu_{\sigma(i)}=\alpha$ if $i\in S\setminus S'$. For every remaining pair $\{i,\sigma(i)\}$ put
$\mu_i=\alpha$ and $\mu_{\sigma(i)}=\alpha^{-1}$. Property (i) follows from
Definition~\ref{def:perp}.

All $\mu_j$ with $j\notin\{0,\sigma(0)\}$ are powers of $\alpha$ and so commute with one another.
For $1\le j\le\sigma(0)-1$, either $\sigma(j)$ also lies in this range (and the pair contributes
$\alpha\alpha^{-1}$), or $\sigma(j)>\sigma(0)$, i.e., $j\in S$. Hence
$\prod_{j=1}^{\sigma(0)-1}\mu_j=\alpha^{|S'|-|S\setminus S'|}$, which is $1$ if $|S|$ is even and
$\alpha$ if $|S|$ is odd. Similarly, $\prod_{j=\sigma(0)+1}^{2\ell-1}\mu_j$ is $1$ if $|S|$ is even and
$\alpha^{-1}$ if $|S|$ is odd. Therefore $\prod_{j=1}^{2\ell}\mu_j$ equals $\gamma\beta$ or
$\gamma\alpha^{-1}\beta\alpha$, an $|E|$-cycle in either case.
\end{proof}

Extend the definition by $\mu_i=\mu_{i'}$ where $i'\equiv i\pmod{2\ell}$, $0\le i'\le 2\ell-1$, and
set $\hat\mu_0=\mathrm{id}$ and $\hat\mu_i=\prod_{j=1}^{i}\mu_j$ for $i\ge1$.

\begin{lemma}\label{lem:distinct}
In the setting of Lemma~\ref{lem:mus}, for every $i\ge0$ the elements $\hat\mu_{i+2k\ell}(0)$,
$0\le k\le|E|-1$, are pairwise distinct.
\end{lemma}

\begin{proof}
Let $i'\equiv i\pmod{2\ell}$ with $0\le i'\le 2\ell-1$, and let $P=\prod_{j=i'+1}^{2\ell}\mu_j$.
Then $\prod_{j=i+1}^{i+2\ell}\mu_j=\big(\prod_{j=1}^{i'}\mu_j\big)P
=P^{-1}\big(\prod_{j=1}^{2\ell}\mu_j\big)P$ is a conjugate of an $|E|$-cycle, hence an $|E|$-cycle.
Since $\hat\mu_{i+2(k+1)\ell}=\big(\prod_{j=i+1}^{i+2\ell}\mu_j\big)\hat\mu_{i+2k\ell}$, the points
$\hat\mu_{i+2k\ell}(0)$ are consecutive points of the orbit of an $|E|$-cycle and therefore
distinct for $0\le k\le|E|-1$.
\end{proof}

\begin{theorem}\label{thm:Elift}
Let $E$ be an abelian group with an admissible triple. Every cyclic \Dseq{} of a group $H$ has a
cyclic \Dseq{} lift in $H\times E$.
\end{theorem}

\begin{proof}
Let $|H|=\ell$, let $\bh$ be a cyclic \Dseq{} of $H$ of length $2\ell$, extended periodically
($h_i=h_{i'}$ for $i\equiv i'\pmod{2\ell}$), and let $\sigma$ be the fixed-point-free involution of
$\{0,\dots,2\ell-1\}$ with $\bar h_{\sigma(i)}=\bar h_i$. Let $\mu_i$ be as in Lemma~\ref{lem:mus}.
Put $N=2\ell|E|$ and define, for $0\le i\le N-1$,
\[
t_i=\hat\mu_i(0),\qquad g_i=(h_i,t_i).
\]
Then $t_0=0$ and $t_i=\mu_i(t_{i-1})$. Moreover
$t_{N-1}=\mu_{2\ell}^{-1}\big(\prod_{j=1}^{2\ell}\mu_j\big)^{|E|}(0)=\gamma^{-1}(0)=0$, so
$g_{N-1}=(1_H,0)$. We regard $\bg$ as circular; then $t_{i}=\mu_i(t_{i-1})$ holds for all $i\in\Z_N$
(for $i=0$ because $\mu_0(0)=\gamma(0)=0$), and the circular quotients are
\[
g_{i-1}^{-1}g_i=\big(\bar h_i,\ \mu_i(t_{i-1})-t_{i-1}\big),\qquad i\in\Z_N .
\]
By Lemma~\ref{lem:circular} it suffices to show that each element of $H\times E$ appears twice in
$\bg$ and twice among these circular quotients.

\emph{Terms.} For $0\le i\le 2\ell-1$, Lemma~\ref{lem:distinct} gives
$\{g_{i+2k\ell}:0\le k<|E|\}=\{h_i\}\times E$. As each element of $H$ appears twice in $\bh$, each
element of $H\times E$ appears twice in $\bg$.

\emph{Quotients.} Let $(u,v)\in H\times E$ and let $l$ be an index with $\bar h_l=\bar
h_{\sigma(l)}=u$. Since $\mu_l\perp\mu_{\sigma(l)}$, one of the following holds for some $x,y\in E$:
(1) $x\ne y$ and $\mu_l(x)-x=\mu_l(y)-y=v$; (2) $x\ne y$ and
$\mu_{\sigma(l)}(x)-x=\mu_{\sigma(l)}(y)-y=v$; (3) $\mu_l(x)-x=\mu_{\sigma(l)}(y)-y=v$. By
Lemma~\ref{lem:distinct} (applied to $i=l-1$, or to $i=2\ell-1$ if $l=0$, with indices of $\bt$
read modulo $N$), $t_{l-1+2k\ell}$ runs over all of $E$ as $k$ runs over $\Z_{|E|}$; so in case (1) there are $k\ne k'$ with $t_{l-1+2k\ell}=x$ and
$t_{l-1+2k'\ell}=y$, and the circular quotients at $l+2k\ell$ and $l+2k'\ell$ both equal $(u,v)$.
Case (2) is the same with $\sigma(l)$ in place of $l$. In case (3) choose $k,k'$ with
$t_{l-1+2k\ell}=x$ and $t_{\sigma(l)-1+2k'\ell}=y$; the quotients at the indices $l+2k\ell$ and
$\sigma(l)+2k'\ell$ equal $(u,v)$, and these indices are distinct since $l\not\equiv\sigma(l)
\pmod{2\ell}$. Hence every element appears at least twice, and so exactly twice, among the $N$
circular quotients.
\end{proof}

\begin{example}\label{ex:Z4Z4}
Let $H=\Z_4$ with the cyclic \Dseq{} $\bh:0,1,3,2,2,3,1,0$ (from the sequencing $0,1,3,2$ via
Lemma~\ref{lem:kfold}); then $\bar\bh:0,1,2,3,0,1,2,3$ and $\sigma(i)=i+4\pmod 8$. Here
$S=\{1,2,3\}$; taking $S'=\{1,2\}$ we get $\mu_0=\gamma$, $\mu_4=\beta$, $\mu_1=\mu_2=\mu_7=\alpha$,
$\mu_3=\mu_5=\mu_6=\alpha^{-1}$, with $(\alpha,\beta,\gamma)$ from Lemma~\ref{lem:triple} for $\nu=2$,
and $\mu_8\cdots\mu_1=\gamma\alpha^{-1}\beta\alpha$. The values $t_{i+8k}=\hat\mu_{i+8k}(0)$ are
\[
\begin{array}{c|cccccccc}
k\backslash i&0&1&2&3&4&5&6&7\\\hline
0&0&1&3&1&0&3&1&3\\
1&1&3&0&3&2&2&2&2\\
2&2&2&2&2&3&1&0&1\\
3&3&0&1&0&1&0&3&0
\end{array}
\]
and $g_{i+8k}=(h_i,t_{i+8k})$ is a cyclic \Dseq{} of $\Z_4\times\Z_4$.
\end{example}

\begin{theorem}\label{thm:m}
Let $m\not\equiv 2\pmod 4$. Every cyclic \Dseq{} of a group $H$ has a cyclic \Dseq{} lift in
$H\times\Z_m$.
\end{theorem}

\begin{proof}
Write $m=2^\nu m'$ with $m'$ odd; then $\nu\ne1$. If $\nu=0$, this is Lemma~\ref{lem:oddlift}. If
$\nu\ge2$, Lemma~\ref{lem:triple} and Theorem~\ref{thm:Elift} give a cyclic \Dseq{} lift in
$H\times\Z_{2^\nu}$, and Lemma~\ref{lem:oddlift} lifts it further to
$H\times\Z_{2^\nu}\times\Z_{m'}\cong H\times\Z_m$; the composite is a lift of the original sequence.
\end{proof}

\begin{proof}[Proof of Theorem~C]
An abelian group with no direct summand of order $2$ is a direct product of cyclic groups
$\Z_{m_1}\times\cdots\times\Z_{m_s}$ with each $m_j$ odd or divisible by $4$ (e.g.\ its primary
decomposition). Apply Theorem~\ref{thm:m} $s$ times.
\end{proof}

\section{\texorpdfstring{Direct factors $(\Z_2)^k$}{Direct factors of order 2}}\label{sec:Z2}

Throughout this section $E=(\Z_2)^k$, written additively; note that $-x=x$ in $E$.

\subsection{Odd sequenceable groups}

\begin{lemma}\label{lem:oddseq}
Let $H$ be a sequenceable group of odd order. Then $H\times(\Z_2)^k$ has a cyclic \Dseq{} for every
$k\ge0$.
\end{lemma}

\begin{proof}
For $k=0$ use Lemma~\ref{lem:kfold}. Let $k\ge1$, $|H|=n$, $N=2^{k+1}$, let $\bh$ be a sequencing of
$H$, and let $d_0,\dots,d_{N-1}$ be a cyclic \Dseq{} of $E$ (Theorem~\ref{thm:abelian}); put
$d_N=d_0=0$. For $0\le i\le Nn-1$ write $i=qn+j$ with $0\le j\le n-1$ and define
\[
p_i=\begin{cases}(h_j,\ d_{q+(j\bmod 2)}) & q\text{ even},\\
(h_{n-1-j},\ d_{q+(j\bmod 2)}) & q\text{ odd}.\end{cases}
\]
Thus block $q$ runs through $\bh$ forwards ($q$ even) or backwards ($q$ odd), with $E$-components
alternating $d_q,d_{q+1},d_q,\dots,d_q$ (as $n$ is odd, $j=0$ and $j=n-1$ are both even).

\emph{Terms.} Since $n$ is odd, $j$ and $n-1-j$ have the same parity. An element $h_a$ with $a$ even
is paired with $d_q$ for every $q$, and with $a$ odd with $d_{q+1}$ for every $q$; as
$q$ (resp.\ $q+1$) runs over $\Z_N$ each element of $E$ occurs twice. So every element of $H\times E$
appears twice.

\emph{Quotients.} Inside block $q$ the quotients are $(\bar h_j,d_q+d_{q+1})=(\bar h_j,\bar d_{q+1})$
for $q$ even and $(\bar h_{n-j}^{-1},\bar d_{q+1})$ for $q$ odd, $1\le j\le n-1$. Consecutive blocks
meet in equal $H$-components; the last term of block $q-1$ has $E$-component $d_{q-1}$ (as $j=n-1$
is even) and the first term of block $q$ has $E$-component $d_q$, so the quotient at the start of
block $q\ge1$ is $(1_H,\bar d_q)$, while $\bar p_0=(1_H,0)=(1_H,\bar d_0)$. Since $\bd$ is cyclic,
$\bar d_N:=d_{N-1}+d_N=d_{N-1}+d_0=0=\bar d_0$. Hence the quotients with $H$-component $1_H$ are
$(1_H,\bar d_q)$, $0\le q\le N-1$, and, for $u\ne1_H$, the quotients with $H$-component $u$ are
$(u,\bar d_{q+1})$, $0\le q\le N-1$ (because $u$ occurs exactly once among
$\bar h_1,\dots,\bar h_{n-1}$ and exactly once among their inverses). As $\bar\bd$ contains each
element of $E$ twice, so does each of these lists. Finally, the last term is
$p_{Nn-1}=(h_0,d_{N-1})=(1_H,0)$.
\end{proof}

\begin{remark}
The cyclicity of $\bd$ is essential: it is what makes the quotient $\bar d_N$ produced by the last
block coincide with $\bar d_0$. Cyclic \Dseq{}s of $(\Z_2)^k$ exist by Theorem~\ref{thm:abelian}.
For $k=1$, Lemma~\ref{lem:oddseq} is a double-sequencing analogue of Anderson's theorem that
$H\times\Z_2$ is sequenceable whenever $H$ is an odd sequenceable group \cite{And76}.
\end{remark}

\subsection{A block construction}
Let $n=|H|$ be even, $N=2^{k+1}=2|E|$, and let $\bz=(z_0,\dots,z_{2N-1})$ be a sequence in $E$,
indices modulo $2N$. Consider the conditions:
\begin{enumerate}[label=\textup{(F\arabic*)},leftmargin=3em]
\item every element of $E$ occurs exactly twice among $z_i$ with $i\equiv0,3\pmod4$;
\item every element of $E$ occurs exactly twice among $z_i$ with $i\equiv1,2\pmod4$;
\item the elements $s_q=z_{2q}+z_{2q+1}$ are pairwise distinct for even $q$, and pairwise distinct
for odd $q$ $(0\le q\le N-1)$;
\item every element of $E$ occurs exactly twice among the \emph{junction sums}
$z_{2q-1}+z_{2q}$, $q\in\Z_N$;
\item[\textup{(S)}] every element of $E$ occurs exactly twice among $s_q=z_{2q}+z_{2q+1}$,
$0\le q\le N-1$.
\end{enumerate}
Clearly (F3) implies (S).

\begin{lemma}\label{lem:block}
Let $H$ be a group of even order $n$, let $\bh=(h_0,\dots,h_{n-1})$ be an ordering of the elements
of $H$, and let $\bz$ satisfy \textup{(F1)}, \textup{(F2)} and \textup{(F4)}. Define the circular
sequence $\bx$ of length $Nn$ by
\[
x_{qn+j}=\begin{cases}(h_j,\ z_{2q+(j\bmod2)}) & q\text{ even},\\
(h_{n-1-j},\ z_{2q+(j\bmod2)}) & q\text{ odd},\end{cases}\qquad 0\le q\le N-1,\ 0\le j\le n-1 .
\]
Assume that either
\begin{enumerate}[label=\textup{(\alph*)}]
\item $\bh$ is a translate of a sequencing \textup{(}i.e., $\bar h_1,\dots,\bar h_{n-1}$ are the
non-identity elements, each once\textup{)} and $\bz$ satisfies \textup{(S)}; or
\item $\bh$ is a terrace and $\bz$ satisfies \textup{(F3)}.
\end{enumerate}
Then $H\times E$ has a cyclic \Dseq.
\end{lemma}

\begin{proof}
We verify the hypothesis of Lemma~\ref{lem:circular}.

\emph{Terms.} As $n$ is even, $j$ and $n-1-j$ have opposite parity. Hence $h_a$ with $a$ even is
paired with $z_{2q}$ for even $q$ (indices $\equiv0\pmod 4$) and with $z_{2q+1}$ for odd $q$ (indices
$\equiv3\pmod4$); by (F1) it is paired with every element of $E$ exactly twice. Similarly, $h_a$ with
$a$ odd is paired with the $z_i$, $i\equiv1,2\pmod 4$, and (F2) applies.

\emph{Quotients.} Consecutive blocks meet in equal $H$-components ($h_{n-1}$ between an even and an
odd block, $h_0$ between an odd and an even block, including the wrap-around from block $N-1$ to
block $0$). The last term of block $q-1$ has $j=n-1$ odd, so its $E$-component is $z_{2q-1}$; hence
the circular quotient at the start of block $q$ is $(1_H,z_{2q-1}+z_{2q})$. Inside block $q$ the
quotients are $(\bar h_j,s_q)$ for $q$ even and $(\bar h_{n-j}^{-1},s_q)$ for $q$ odd,
$1\le j\le n-1$. No quotient inside a block has $H$-component $1_H$, so by (F4) every element
$(1_H,v)$ occurs exactly twice. For $u\ne1_H$ and $v\in E$, the number of occurrences of $(u,v)$ is
\[
\#\{q\text{ even}:s_q=v\}\cdot\#\{j\ge1:\bar h_j=u\}+\#\{q\text{ odd}:s_q=v\}\cdot
\#\{j\ge1:\bar h_j=u^{-1}\}.
\]
In case (a) both $\#$-factors involving $\bh$ equal $1$, and the sum is $\#\{q:s_q=v\}=2$ by (S). In
case (b) both $\#$-factors involving $\bz$ equal $1$ by (F3), and the sum is
$\#\{j:\bar h_j=u\}+\#\{j:\bar h_j=u^{-1}\}=2$ by the terrace property (for $u^2=1_H$ each summand
is $1$).
\end{proof}

\subsection{Sequences for sequencings}
We now construct sequences satisfying (F1), (F2), (F4) and (S) for every $k\ge1$. For a sequence
$\bc=(c_0,\dots,c_{L-1})$ with $c_0=0$ write $\bar c_0=0$ and $\bar c_i=c_{i-1}+c_i$. If moreover
$c_{L-1}=0$, then the junction sums of $\bc$ are the $\bar c_i$ with $i$ even, and the sums $s_q$
are the $\bar c_i$ with $i$ odd.

\begin{definition}\label{def:admissible}
Let $k\ge2$ and $L=2^{k+2}$. Call positions $i\equiv0,3\pmod4$ of \emph{type I} and $i\equiv1,2\pmod 4$
of \emph{type II}. A sequence $\bc=(c_0,\dots,c_{L-1})$ in $(\Z_2)^k$ is \emph{admissible} if
$c_0=c_{L-1}=0$ and:
\begin{enumerate}[label=\textup{(A\arabic*)},leftmargin=3em]
\item for every $i$ there is a unique $j=\theta(i)\ne i$ of the same type with $c_j=c_i$, and every
element of $(\Z_2)^k$ occurs exactly twice among positions of each type;
\item if $i\equiv0,7\pmod 8$ then $\theta(i)\equiv0,7\pmod 8$; otherwise $\theta(i)\equiv
i\pmod 4$ and $\theta(i)\not\equiv0,7\pmod 8$;
\item for every $i$ there is a unique $j\ne i$ with $j\equiv i\pmod 2$ and $\bar c_j=\bar c_i$, and
every element occurs exactly twice among $\bar c_i$, $i$ even, and twice among $\bar c_i$, $i$ odd;
\item $\bar c_i=\bar c_{i+1}=\bar c_{i+L/2}=\bar c_{i+1+L/2}$ for every $i\equiv0\pmod 8$ with
$0\le i<L/2$.
\end{enumerate}
\end{definition}

Note that (A2) implies: if $i\equiv 3$ (resp.\ $4$) $\pmod 8$ then $\theta(i)\equiv3$ (resp.\ $4$)
$\pmod 8$, and if $i\equiv 1,3,5\pmod 8$ (resp.\ $2,4,6$) then so is $\theta(i)$. By (A3), (A4) and
uniqueness, the indices $i\equiv0,1\pmod8$ are paired among themselves in (A3).

\begin{lemma}\label{lem:induction}
If $\bc$ is an admissible sequence in $(\Z_2)^k$ of length $L=2^{k+2}$, then the sequence
$\bd=(d_0,\dots,d_{2L-1})$ in $(\Z_2)^{k+1}=(\Z_2)^k\times\Z_2$ defined below is admissible. Let
\[
\eta_i=\begin{cases}
0 & \text{if } i \text{ is odd and } i<L-1,\\
0 & \text{if } i\equiv2,4,6\pmod 8\text{ and } i<\theta(i),\\
0 & \text{if } i\equiv0\pmod8 \text{ and } i<L/2,\\
1 & \text{otherwise},
\end{cases}\qquad
\varepsilon_i=\begin{cases}0& i\equiv 2,4,6\pmod8,\\ 1&\text{otherwise},\end{cases}
\]
for $0\le i\le L-1$, and $d_i=(c_i,\eta_i)$, $d_{i+L}=(c_i,\eta_i+\varepsilon_i)$ for $0\le i\le L-1$.
\end{lemma}

\begin{proof}
Since $L\equiv0\pmod8$, $i$ and $i+L$ have the same residue modulo $8$. We have $d_0=(0,0)$ and
$d_{2L-1}=(0,\eta_{L-1}+1)=(0,0)$ because $\eta_{L-1}=1$.

\emph{(A1), (A2).} Let $0\le i\le L-1$ and $v=c_i$. By (A1) for $\bc$, the positions of $\bd$ of the
type of $i$ with first component $v$ are exactly $i,\theta(i),i+L,\theta(i)+L$.
\begin{itemize}
\item If $i\equiv1,3,5\pmod8$, then so is $\theta(i)$, and both are odd and $<L-1$; hence
$d_i=d_{\theta(i)}=(v,0)$ and $d_{i+L}=d_{\theta(i)+L}=(v,1)$.
\item If $i\equiv2,4,6\pmod 8$, then so is $\theta(i)$; exactly one of $\eta_i,\eta_{\theta(i)}$ is
$0$, and $\varepsilon=0$ at these positions; hence $d_i=d_{i+L}\ne d_{\theta(i)}=d_{\theta(i)+L}$.
\item If $i\equiv0,7\pmod8$, then so is $\theta(i)$ and $\varepsilon_i=\varepsilon_{\theta(i)}=1$;
hence $d_i\ne d_{i+L}$ and $d_{\theta(i)}\ne d_{\theta(i)+L}$, so two of the four terms have
second component $0$ and two have $1$.
\end{itemize}
In each case every element $(v,b)$ occurs exactly twice among the four positions, and the partner
of each position satisfies (A2) (it is congruent to it modulo $8$ in the first two cases, and lies in
$\{0,7\}\bmod 8$ in the third).

\emph{(A3), (A4).} Write $\bar d_i=(\bar c_i,\zeta_i)$ for $0\le i\le L-1$, where $\zeta_0=0$ and
$\zeta_i=\eta_{i-1}+\eta_i$. For $1\le i\le L-1$ we have
$\bar d_{i+L}=(\bar c_i,\zeta_i+\varepsilon_{i-1}+\varepsilon_i)$, and $\varepsilon_{i-1}+\varepsilon_i=0$
if and only if $i\equiv0,1\pmod8$. At the seam,
$\bar d_L=(c_{L-1}+c_0,\eta_{L-1}+\eta_0+\varepsilon_0)=(0,0)=(\bar c_0,\zeta_0)$.

\emph{Case $i\not\equiv0,1\pmod 8$.} Then $\bar d_i,\bar d_{i+L}$ have first component $\bar c_i$ and
different second components. The (A3)-partner $j$ of $i$ in $\bc$ also satisfies
$j\not\equiv0,1\pmod 8$, so $\bar d_j,\bar d_{j+L}$ also have first component $\bar c_i$ and different
second components. Thus each of $(\bar c_i,0),(\bar c_i,1)$ occurs exactly twice among indices of the
parity of $i$.

\emph{Case $i\equiv0,1\pmod 8$.} Fix $i_0\equiv0\pmod 8$ with $0\le i_0<L/2$. By (A4) the eight indices
$i_0,i_0+1,i_0+L/2,i_0+1+L/2$ and their translates by $L$ all have first component $\bar c_{i_0}$. One
computes $\zeta_{i_0}=0$ (for $i_0>0$: $\eta_{i_0-1}=\eta_{i_0}=0$), $\zeta_{i_0+1}=\eta_{i_0}+\eta_{i_0+1}=0$,
$\zeta_{i_0+L/2}=\eta_{i_0+L/2-1}+\eta_{i_0+L/2}=0+1=1$ and
$\zeta_{i_0+L/2+1}=1+0=1$ (here $L/2\equiv0\pmod 8$ since $k\ge2$, and $i_0+L/2+1<L-1$). The
translates by $L$ have the same second components (including $\bar d_L=(0,0)$ for $i_0=0$). Hence among
the even indices $i_0,i_0+L/2,i_0+L,i_0+3L/2$ the second components are $0,1,0,1$, and likewise for
the odd indices; this proves (A3) for $\bd$, and also gives
$\bar d_{i}=\bar d_{i+1}=\bar d_{i+L}=\bar d_{i+1+L}$ for all $i\equiv0\pmod8$ with $0\le i<L$
(using (A4) for $\bc$ when $L/2\le i<L$), which is (A4) for $\bd$.
\end{proof}

\begin{corollary}\label{cor:Sseq}
For every $k\ge1$ there is a sequence $\bc$ of length $2^{k+2}$ in $(\Z_2)^k$ with $c_0=c_{L-1}=0$
satisfying \textup{(F1), (F2), (F4)} and \textup{(S)}.
\end{corollary}

\begin{proof}
For $k=1$ take $\bc:0,0,1,1,1,0,1,0$. For $k=2$, writing elements of $(\Z_2)^2$ as strings, the
sequence
\begin{equation}\label{eq:c2}
\bc:\ 00,00,10,11,10,00,11,01,01,01,10,11,10,01,11,00
\end{equation}
is admissible, with
$\bar\bc: 00,00,10,01,01,10,11,10,00,00,11,01,01,11,10,11$. Lemma~\ref{lem:induction} then gives
admissible sequences for all $k\ge2$. Finally (A1) is (F1)--(F2), and since $c_{L-1}=c_0=0$, (A3) gives
(F4) (even indices) and (S) (odd indices).
\end{proof}

\begin{example}\label{ex:k3}
Applying Lemma~\ref{lem:induction} to \eqref{eq:c2} gives
$\eta=(0,0,0,0,0,0,0,0,1,0,1,0,1,0,1,1)$ and the admissible sequence in $(\Z_2)^3$ (each $d_i$ written
as $c_i$ followed by its last bit):
\[
\begin{gathered}
000,000,100,110,100,000,110,010,011,010,101,110,101,010,111,001,\\
001,001,100,111,100,001,110,011,010,011,101,111,101,011,111,000.
\end{gathered}
\]
\end{example}

\subsection{Frames}
For terraces, Lemma~\ref{lem:block}(b) requires (F3) instead of (S).

\begin{definition}\label{def:frame}
A \emph{$k$-frame} is a circular sequence $\bz=(z_0,\dots,z_{2^{k+2}-1})$ in $(\Z_2)^k$ satisfying
(F1)--(F4).
\end{definition}

Frames are invariant under translation $z_i\mapsto z_i+x$, under rotation of indices by multiples of
$4$, and under automorphisms of $(\Z_2)^k$. The sequences of Corollary~\ref{cor:Sseq} are not frames
in general; for instance \eqref{eq:c2} has $s_0=s_4=00$ with $q=0,4$ both even.

\begin{proposition}\label{prop:frames}
$k$-frames exist for $1\le k\le 6$.
\end{proposition}

\begin{proof}
Identify $(\Z_2)^k$ with $\{0,1,\dots,2^k-1\}$ via binary expansion (addition is bitwise XOR). The
sequences in Table~\ref{tab:frames} are frames for $k\le4$; conditions (F1)--(F4) are finitely many
equalities that can be checked by hand or by computer. Frames for $k=5,6$ (of lengths $128$ and $256$)
are given in the ancillary file, together with a program verifying (F1)--(F4).
\end{proof}

\begin{table}[ht]
\centering
\small
\begin{tabular}{@{}cl@{}}
\toprule
$k$ & $k$-frame $z_0,z_1,\dots$\\\midrule
$1$ & $0,0,0,1,0,1,1,1$\\[2pt]
$2$ & $0,0,1,2,\ 1,2,2,3,\ 1,3,3,3,\ 0,1,0,2$\\[2pt]
$3$ & $0,0,2,4,\ 2,4,5,7,\ 3,6,6,6,\ 1,2,1,5,\ 0,1,3,4,\ 2,5,4,7,\ 3,7,7,6,\ 1,3,0,5$\\[2pt]
$4$ & $0,0,4,8,\ 4,8,10,14,\ 6,12,12,12,\ 3,4,2,11,\ 1,3,7,9,\ 5,10,9,15,\ 6,14,15,13,\ 3,7,1,10,$\\
 & $0,13,14,11,\ 2,9,8,9,\ 7,1,6,14,\ 5,6,3,12,\ 1,15,13,10,\ 2,11,11,8,\ 7,2,5,15,\ 4,5,0,13$\\
\bottomrule
\end{tabular}
\caption{$k$-frames for $k\le4$; elements of $(\Z_2)^k$ are written as integers in binary.}
\label{tab:frames}
\end{table}

The frames were obtained recursively: given a $k$-frame $\bz$ of length $L$ and $s\in\Z_{L/4}$, let
$\bz'$ be the sequence $z'_i=z_{i+4s}+z_0+z_{4s}$. The concatenation $\bz\bz'$ (of length $2L$)
satisfies the analogues of (F1)--(F4) in $(\Z_2)^k$ with every multiplicity doubled, the translation by
$z_0+z_{4s}$ ensuring that the two junction sums at the seams agree with the two junction sums of $\bz$
and $\bz'$ that are destroyed by the concatenation. One then seeks
$\varepsilon\in(\Z_2)^{2L}$ such that $\big((\bz\bz')_i,\varepsilon_i\big)$ is a $(k+1)$-frame; after
pairing up equal entries, this becomes a system of linear equations over $\mathbb F_2$. In all our
experiments a solution existed for a suitable $s$, which leads us to the following conjecture.

\begin{conjecture}\label{conj:frames}
For every $k\ge1$ there exists a $k$-frame.
\end{conjecture}

\section{Proofs of Theorems B and D}\label{sec:proofs}

\begin{proposition}\label{prop:HZ2k}
Let $H$ be a group and $k\ge0$. Then $H\times(\Z_2)^k$ has a cyclic \Dseq{} in each of the following
cases: \textup{(i)} $H$ is sequenceable; \textup{(ii)} $H$ is terraced and a $k$-frame exists (in
particular, if $k\le6$).
\end{proposition}

\begin{proof}
For $k=0$ use Lemma~\ref{lem:kfold} and Proposition~\ref{prop:terrace}. Let $k\ge1$. If $|H|$ is odd,
in case (i) apply Lemma~\ref{lem:oddseq}; in case (ii) note that $H\times\Z_2$ is a solvable binary group
other than $Q_8$, hence sequenceable \cite{AI93}, and apply case (i) to $(H\times\Z_2)\times(\Z_2)^{k-1}$.
If $|H|$ is even, apply Lemma~\ref{lem:block}(a) with a sequencing $\bh$ and the sequence of
Corollary~\ref{cor:Sseq} in case (i), and Lemma~\ref{lem:block}(b) with a terrace $\bh$ and a $k$-frame
in case (ii).
\end{proof}

\begin{proof}[Proof of Theorem~B]
Write $K\cong(\Z_2)^k\times K'$, where $K'$ has no direct summand of order $2$.

If $H$ is sequenceable, $H\times(\Z_2)^k$ has a cyclic \Dseq{} by Proposition~\ref{prop:HZ2k}(i), and
Theorem~C (applied to $H\times(\Z_2)^k$ and $K'$) gives a cyclic \Dseq{} of $H\times K$.

Let $|H|$ be odd. If $|K|$ is odd, then $H\times K$ has odd order and Corollary~\ref{cor:odd} applies.
Otherwise $K\cong\Z_{2^\nu}\times K_1$ with $\nu\ge1$. The group $H\times\Z_{2^\nu}$ is solvable (by the
Feit--Thompson theorem), binary, and not isomorphic to $Q_8$ (its Sylow $2$-subgroup is cyclic), so it
is sequenceable \cite{AI93}; the previous case applied to $H\times\Z_{2^\nu}$ and $K_1$ completes the
proof.
\end{proof}

\begin{proof}[Proof of Theorem~D]
By Proposition~\ref{prop:HZ2k}(ii), $H\times(\Z_2)^k$ has a cyclic \Dseq; apply Theorem~C with $K'$.
\end{proof}

\begin{corollary}
For $H\in\{D_6,D_8,Q_8\}$ and every abelian group $K$ with at most six direct summands of order $2$,
the group $H\times K$ has a cyclic \Dseq. If Conjecture~\ref{conj:frames} holds, then $H\times K$ has
a cyclic \Dseq{} for every terraced group $H$ and every abelian group $K$.
\end{corollary}

\section{Concluding remarks and open problems}\label{sec:open}

\begin{enumerate}[leftmargin=2em]
\item By Remark~\ref{rem:noZ2}, the lifting method of Section~\ref{sec:lift} cannot handle $\Z_2$.
\begin{question}
Does every cyclic \Dseq{} of a group $H$ have a cyclic \Dseq{} lift in $H\times\Z_2$?
\end{question}
An affirmative answer, together with Theorem~C, would show that $H\times K$ has a cyclic \Dseq{} for
every group $H$ with a cyclic \Dseq{} and every abelian group $K$. One might hope to circumvent
$\Z_2$ by lifting directly by groups such as $(\Z_2)^2$ or $\Z_6$ via Theorem~\ref{thm:Elift}.
However, an exhaustive computer search shows that neither $(\Z_2)^2$ nor $\Z_6$ has an admissible
triple; in fact $\Z_6$ has no permutation $\alpha$ with $\alpha\perp\alpha^{-1}$ at all. So a
different idea is needed.

\item Conjecture~\ref{conj:frames} would remove the restriction $k\le6$ in Theorem~D.

\item Lemma~\ref{lem:block} shows that terraces and sequencings behave differently under the block
construction; it would be interesting to know whether a terrace of $H$ always gives rise to a
sequence satisfying the hypotheses of Lemma~\ref{lem:block}(a) for some other ordering of $H$.

\item By Corollary~\ref{cor:status}, the double sequencing conjecture is open only for finitely many
orders, but the bound $n_0$ coming from \cite{MP25} is not explicit. The smallest order for which the
conjecture is not known to hold by the results above is $64$.
\end{enumerate}

\end{document}